\documentclass[11p]{article}

\usepackage{amsfonts}
\usepackage{amsmath}
\usepackage{amssymb}
\usepackage{amscd}
\usepackage{amsthm}

\newtheorem{theorem}{Theorem}[section]
\newtheorem{lemma}[theorem]{Lemma}
\theoremstyle{definition}
\newtheorem {definition}[theorem]{Definition}

\theoremstyle{remark}
\newtheorem*{remark}{Remark}

\def\be{\begin{equation}}
\def\ee{\end{equation}}
\def\ba{\begin{eqnarray*}}
\def\ea{\end{eqnarray*}}
\def\bae{\begin{eqnarray}}
\def\eae{\end{eqnarray}}
\def\bc{\begin{center}}
\def\ec{\end{center}}

\begin{document}
\title{A note on
eigenvalues  of random block Toeplitz matrices with slowly growing bandwidth}

\date{March 1, 2011}
\author{Yi-Ting Li, Dang-Zheng Liu, Xin Sun and
Zheng-Dong Wang\\
School of Mathematical Sciences\\
Peking University\\
Beijing, 100871, P. R. China }

\maketitle

\begin{abstract}
This paper can be thought of as a remark of \cite{llw}, where the authors
studied the eigenvalue distribution $\mu_{X_N}$ of random block
  Toeplitz band matrices with given block order $m$.  In this note we will give explicit density functions of $\lim\limits_{N\to\infty}\mu_{X_N}$ when the
bandwidth grows slowly. In fact, these densities are exactly the
normalized one-point correlation functions of $m\times m$ Gaussian
unitary ensemble (GUE for short). The series
$\{\lim\limits_{N\to\infty}\mu_{X_N}|m\in\mathbb{N}\}$ can be seen
as a transition from the standard normal distribution to semicircle
distribution. We also show a similar relationship between GOE and
block Toeplitz band matrices with symmetric blocks.

\end{abstract}

\section{Introduction}

A block Toeplitz matrix is a block matrix which can be written as
\[
T_N=(A_{i-j})_{i,j=1}^N=
\begin{array}{ccc}
\left( \begin{array}{cccccc}
A_{0}&A_{-1}&A_{-2}&\cdots&A_{-(N-1)}\\
A_{1}&A_{0}&A_{-1}&\cdots&A_{-(N-2)}\\
A_{2}&A_{1}&A_{0}&\cdots&A_{-(N-3)}\\
\vdots&\vdots&\vdots&\ddots &\vdots\\
A_{N-1}&A_{N-2}&A_{N-3}&\cdots&A_{0}
\end{array}\right)
\end{array}
\]
where $\{A_{-(N-1)},...,A_0,...,A_{N-1}\}$ are $m\times m$ matrices and
$\,A_s=(a_{ij}(s))_{i,j=1}^m$. If the $a_{ij}(s)$'s are
random variables, then we call $T_N$ a random block Toeplitz matrix.
We suppose the $a_{ij}(s)$'s are real random variables and $A_s=(A_{-s})^T$. Besides, we list some assumptions as follows.\\
Independence of the elements:\\
(1) $a_{i_1j_1}(s_1)$ and $a_{i_2j_2}(s_2)$ are independent if
$|s_1|\ne|s_2|$, \\
(2) If $s\ne0$ and $(i_1,j_1)\ne(i_2,j_2)$ then $a_{i_1j_1}(s)$ and
$a_{i_2j_2}(s)$ are independent, \\
(3) If $(i_1,j_1)\ne(i_2,j_2)$ and $(i_1,j_1)\ne(j_2,i_2)$ then
$a_{i_1j_1}(0)$ and $a_{i_2j_2}(0)$ are independent.\\
Uniform boundness condition:\\
(4) $\mathbb{E}[a_{ij}(s)]=0$, $\mathbb{E}[|a_{ij}(s)|^2]=1$, $-(N-1)\le s\le N-1$,
$1\le i,j\le m$\\
and
\[
\sup\limits_{N\in\mathbb{N}\atop-(N-1)\le s\le
N-1}\{\mathbb{E}[|a_{ij}(s)|^k]|1\le i,j\le m\}=C_{k,m}<+\infty
\]
and
\[
\sup\limits_{m\in\mathbb{N}}C_{k,m}=C_k<+\infty.
\]
Slowly growing bandwidth condition:\\
(5)The block Toeplitz matrix is a band block matrix
with bandwidth $b_N$, that is,  $A_s=0$ for $|s|>b_N$. Moreover, $b_N$
satisfies: $\lim\limits_{N\to\infty}b_N=\infty$ and $b_N=o(N)$.

Gaussian unitary ensemble (GUE for short) $(\mathcal{H}_{m},d \mu)$
is the space $\mathcal{H}_{m}$ of all Hermitian $(m \times
m)$-matrices with a certain Gaussian measure $d \mu$ (see
\cite{agz,m}). We will regularly use the notations
$\langle\cdot\rangle_{\textmd{GUE}}$ and
$\langle\cdot\rangle_{\textmd{TBM}}$, respectively denoting the
expectations under GUE and random Toeplitz band matrices. If $H =
(h_{ij})$ is a matrix from GUE, then it is easy to see $\langle
h_{ij}h_{ji}\rangle_{\textmd{GUE}}=1$ while all other second moments
are equal to zero: $\langle h_{ij}h_{kl}\rangle_{\textmd{GUE}}=0$,
whenever $(i,j)\neq(l,k)$.

If $A=(a_{ij}(\omega))_{i,j=1}^N$ is a real symmetric or complex
Hermitian random matrix and its entries are random variables on a
probability space $\Omega$ with a probability measure $P$, then the
eigenvalue distribution of $A$ is
\[
\mu_{_A}=\frac{1}{N}\int_{\Omega}\sum\limits_{j=1}^N\delta_{\lambda_j(\omega)}dP(\omega)
\]
where $\lambda_j(\omega)$'s are the $N$ real eigenvalues of $A$. We
have

\begin{theorem}\label{thm1}
Let $T_N$ be an $mN\times mN$ random block Toeplitz matrix as above. Set $X_N=\frac{T_N}{\sqrt {2mb_N}}$, 
then $\mu_{X_N}$   converges weakly to $f_m(x)d x$
  as $N\to\infty$. Moreover, if the bandwidth satisfies
  $\sum_{N=1}^{\infty} b_{N}^{-2} <\infty,$
  then the convergence is almost sure. 
Here
$f_m(x)=\frac{1}{\sqrt{m}}\sum\limits_{j=0}^{m-1}\psi_j^2(\sqrt{m}x)$
where $\psi_j$ is the $j$th nomarlized oscillator wave-function:
$\psi_j(x)=\frac{e^{-\frac{x^2}{4}}H_j(x)}{\sqrt{\sqrt{2\pi}j!}}$
and $H_j$ is the $j$th Hermite polynomial:$ H_j(x)=(-1)^j
e^{\frac{x^2}{2}}\frac{d^j}{dx^j}e^{-\frac{x^2}{2}}.$
\end{theorem}
\begin{remark}$f_m(x)$ is the one-point
correlation function of GUE up to the scaling \cite{m}. As $f_1$ is
the density function of the standard normal distribution and when
$m\to\infty$ $f_m$ converges in law to the semicircle distribution
(see \cite{llw}),
$\{\lim\limits_{N\to\infty}\mu_{X_N}|m\in\mathbb{N}\}$ can be seen
as a transition from $N(0,1)$ to semicircle distribution.
\end{remark}

In \cite{kkm}, Kologlu, Kopp and Miller got a very similar result.
They proved that the limiting eigenvalue density of symmetric block
circulant Toeplitz ensemble is the same as the eigenvalue density of
GUE (Theorem 1.4(1) of \cite{kkm}). But there is an essential
difference between their model and our model: In \cite{kkm} the
block Toeplitz matrix has a circulant structure while in this paper
the block Toeplitz matrix has a band structure. In the viewpoint of
combinatorics, both of the two structures are strong. So they are
two different models induced from the block Toeplitz matrix model
and interestingly have the same limiting eigenvalue distribution.
There is also another difference: In \cite{kkm} the entries of the
block Toeplitz matrix are i.i.d random variables but in our model
the entries only have to satisfy the independent condition.

\begin{theorem} \label{tracepolynomial}Let $\nu_{1},\ldots,\nu_{r}$ be nonnegetive integers, and $\nu=\nu_{1}+\cdots+ \nu_{r}$.
 Set $Y_N=\frac{T_N}{\sqrt {2b_N}}$, then
\begin{equation}\label{traceproduct}\lim\limits_{N\to\infty}\frac{1}{N^{\nu}}\langle\prod_{i=1}^r(\mathrm{tr}Y_N^{i})^{\nu_i}\rangle_{\mathrm{TBM}}
=\langle\prod_{i=1}^r(\mathrm{tr}H^{i})^{\nu_i}\rangle_{\mathrm{GUE}}.\end{equation}

\end{theorem}






\section{Proof of Main Results}

\begin{definition}
Let $[n]=\{1,2,...,n\}$, $\forall n\in\mathbb{N}$. \\
(1) We call $\pi =\{\{a_1,b_1\},...,\{a_k,b_k\}\}$ a pair partition
of $[2k]$ if $\bigcup\limits_{j=1}^r\{a_j,b_j\}=[2k]$ and
$\{a_i,b_i\}\bigcap \{a_j,b_j\}=\emptyset$ if
$i\ne j$. For convenience, we assume that $a_1<\cdots<a_k$ and $a_i<b_i$ for each pair, under which such a pair partition is called a Wick coupling. For $\pi$, we define $\pi(a_j)=b_j$ and $\pi(b_j)=a_j$ $(1\le j\le k)$. We denote by $\mathcal{P}_2(2k)$   the set of pair partitions of $[2k]$.\\
(2) Suppose $\pi\in\mathcal{P}_2(2k)$. Then $\pi$ can be seen as a
permutation: $(a_1,b_1) \cdots (a_k,b_k)$.
Consider the canonical cycle $\gamma_0=(1,2,...,2k-1,2k)$. Let
$g(\pi)$  denote the number of orbits of the permutation
$\gamma_0\circ\pi$.
\end{definition}
The following lemma is a well-known consequence of Wick's formula on
the moments of GUE. One can get it with the method of moment
generating function (see Section 3.3.1 of \cite{agz}).
\begin{lemma}\label{momentofGUE}
Suppose $H=(b_{ij})_{i,j=1}^m$ is an $m\times m$ random Hermitian
matrix from GUE. Set $Y=\frac{H}{\sqrt{m}}$. Then the odd moments of
$\mu_Y$ are all 0 and the even moments of $\mu_Y$ are
$m_{2k}(\mu_Y)=\sum\limits_{\pi\in\mathcal{P}_2(2k)}m^{g(\pi)-k-1}.$
\end{lemma}

\begin{proof}[Proof of Theorem \ref{thm1}]
From Theorem 4.3 of \cite{llw} and Lemma \ref{momentofGUE}, we know
that $\mu_Y$ and $\gamma_T^m$ have the same moments. It follows
from Carleman's Theorem (see \cite{f}) that $\gamma_T^m$ should be
$\mu_Y$. So the density function of $\gamma_T^m$ is the one-point
correlation function of GUE, which is a well-known function
\cite{agz,m}:
\[
f_m(x)=\frac{1}{\sqrt{m}}\sum\limits_{j=0}^{m-1}\psi_j^2(\sqrt{m}x).
\]
From \cite{bb} we know if $\sum_{N=1}^{\infty} b_{N}^{-2} <\infty$,
then the convergence is almost sure.
\end{proof}

\begin{proof}[Proof of Theorem \ref{tracepolynomial}]
By Lemma 2.2 of \cite{llw}, the main contribution in the expansion
of $
\frac{1}{N^{\nu}}\langle\prod_{i=1}^r(\mathrm{tr}Y_N^{i})^{\nu_i}\rangle_{\mathrm{TBM}}$
comes from the pair partitions. As in the proof of Theorem 4.3 in
\cite{llw}, we can show that
\begin{eqnarray*}
 \lim\limits_{N\to\infty}\frac{1}{N^{\nu}}\langle\prod_{i=1}^r(\mathrm{tr}Y_N^{i})^{\nu_i}\rangle_{\mathrm{TBM}}=
\begin{cases}
0&\text{if }\sum\limits_{i=1}^r i \nu_i\text{ is odd}\\
\sum\limits_{\pi\in\mathcal{P}_2(\sum\limits_{i=1}^r i \nu_i)}m^{F(\pi)}&\text{if
}\sum\limits_{i=1}^r  i \nu_i\text{ is even}
\end{cases}.
\end{eqnarray*}
Here the definition of $F(\pi)$ as follows:
Suppose that $\sum\limits_{i=1}^ri \nu_i=2 \tilde{\nu}$ is even and $\pi=\{\{a_1,b_1\},\ldots,\{a_{\tilde{\nu}},b_{\tilde{\nu} }\} \}\in\mathcal{P}_2(2\tilde{\nu})$,
then $F(\pi)$ denotes the number   of ``free indices" of the system
$\begin{cases}
t_{a_i}=t_{f(b_1)}\\
t_{b_i}=t_{f(a_i)}
\end{cases}
(1\le i\le\tilde{\nu}),\text{ where } f \text{ is defined as below:
}$
Set $\nu_0=0$. For  $1\le x\le
\sum\limits_{i=1}^r i\nu_i$, if $\exists\ 0\le s\le r-1$ and $1\le
a\le \nu_{s+1}$ such that $x=\sum\limits_{i=0}^s i\nu_i+(s+1) a$,
then $f(x)=\sum\limits_{i=0}^s i\nu_i +(s+1)(a-1)+1$; otherwise
$f(x)=x+1.$

By Wick's formula, we can compute the integral
$\langle\prod_{i=1}^r(\mathrm{tr}H^{i})^{\nu_i}\rangle_{\mathrm{GUE}}$
and complete the proof.

\end{proof}


\section{Similar Results for GOE and Block Toeplitz Matrix with Symmetric Blocks}

First, we remark that we can get the same results associated with  GUE if one of the following conditions are imposed:\\
 1) each
 block of $T_N$ is a complex matrix and $A_{-s} =(\overline{A_{s}})^{T}$;\\
 2) each block of $T_N$ is a Hermitian matrix, that is,
$A_{-s} =A_{s}=(\overline{A_{s}})^{T}$.

 However, the situation becomes different when each
 block of $T_N=(A_{i-j})_{i,j=1}^N$ is a symmetric matrix, that is, $A_{-s} =A_{s}=( A_{s} )^{T}$. Besides, we modify the second moments of each block as follows:
$\mathbb{E}[|a_{ij}(s)|^2]=
\begin{cases}1 &\text{ if } i\ne j\\
2 &\text{ if } i=j\end{cases},\forall s.$


We still use the notations $\langle\cdot\rangle_{\textmd{GOE}}$ and
$\langle\cdot\rangle_{\textmd{TBM}}$, respectively denoting the
expectations under GOE and Toeplitz band matrices with symmetric
blocks.

\begin{theorem}\label{GOE}
Let $T_N$ be an $mN\times mN$ random block Toeplitz matrix as
mentioned above. Set $X_N=\frac{T_N}{\sqrt {2mb_N}}$, then
\begin{equation} \lim\limits_{N\to\infty}\frac{1}{mN}\langle  \mathrm{tr}X_N^{k}  \rangle_{\mathrm{TBM}}
=\frac{1}{m}\langle \mathrm{tr}(H/\sqrt{m})^{k} \rangle_{\mathrm{GOE}}.\end{equation}
Moreover,  $\mu_{X_N}$   converges weakly to $g_m(x)d x$
  as $N\to\infty$. And if the bandwidth satisfies
  $\sum_{N=1}^{\infty} b_{N}^{-2} <\infty,$
  then the convergence is almost sure.
  Here
\begin{equation}\label{goedensity}
g_m(x)=\frac{1}{\sqrt{m}}\sum\limits_{j=0}^{m-1}\psi_j^2(\sqrt{m}x)+(\frac{m}{2})^{1/2}\psi_{m-1}(\sqrt{m}x)\int_{-\infty}^{\infty}
\varepsilon(x-t)\psi_m(\sqrt{m}t)d t+\alpha_{m}(x),
\end{equation}
$\varepsilon(x)=\frac{1}{2}\mathrm{sign}(x)$  and  $\psi_j$ is the $j$th nomarlized oscillator wave-function as in Theorem \ref{thm1}, while
\begin{eqnarray*}
\alpha_{m}(x)=
\begin{cases}\frac{1}{m}\psi_{2s}(\sqrt{m}x)\div
\int_{-\infty}^{\infty} \psi_{2s}(\sqrt{m}t)d t&\text{if }m=2s+1\\
0&\text{if }m=2s\end{cases}.
\end{eqnarray*}

\end{theorem}

\begin{proof}[Proof of Theorem \ref{GOE}]
For $H=(h_{ij})_{i,j=1}^m$, let $Y=\frac{H}{\sqrt{m}}$. From Wick's
formula we know $\frac{1}{m}\langle \mathrm{tr}(H/\sqrt{m})^{k}
\rangle_{\mathrm{GOE}}=0$ when $k$ is odd and $\frac{1}{m}\langle
\mathrm{tr}(H/\sqrt{m})^{2k}\rangle_{\mathrm{GOE}}$ is
\begin{eqnarray*}
m^{-k-1}\sum\limits_{t_1,...,t_{2k}=1}^m\sum \big(\langle
h_{t_{a_1}t_{a_1+1}}h_{t_{b_1}t_{b_1+1}}\rangle_{\mathrm{GOE}}
\cdots \langle
h_{t_{a_k}t_{a_k+1}}h_{t_{b_k}t_{b_k+1}}\rangle_{\mathrm{GOE}}\big)
\end{eqnarray*}
where the second sum is taken over all
$\pi=\{\{a_1,b_1\},...,\{a_k,b_k\}\}\in\mathcal{P}_2(2k)$ and
$t_{2k+1}:=t_1$. $H$ is symmetric, so $ \langle
h_{t_{a_1}t_{a_1+1}}h_{t_{b_1}t_{b_1+1}}\rangle_{\mathrm{GOE}}
\cdots \langle
h_{t_{a_k}t_{a_k+1}}h_{t_{b_k}t_{b_k+1}}\rangle_{\mathrm{GOE}}\ne 0$
implies $\begin{cases}t_{a_i}=t_{b_i+1}\\ t_{b_i}=t_{a_i+1}
\end{cases}
\text{or}
\begin{cases}t_{a_i}=t_{b_i}\\ t_{a_i+1}=t_{b_i+1}
\end{cases}
\mathrm{for  } \  1\le i \le k$, moreover, each term
\begin{eqnarray*}
\langle
h_{t_{a_i}t_{a_i+1}}h_{t_{b_i}t_{b_i+1}}\rangle_{\mathrm{GUE}}=
\begin{cases}2&\text{if }t_{a_i}=t_{b_i}=t_{a_i+1}=t_{b_i+1}\\
1&\text{otherwise}
\end{cases}
\end{eqnarray*}
For a given
$\pi=\{\{a_1,b_1\},...,\{a_k,b_k\}\}\in\mathcal{P}_2(2k)$, set
\begin{eqnarray*}
A(\pi)=\Big\{(t_1,...,t_{2k})\in[m]^{2k}\Big|\begin{cases}t_{a_i}=t_{b_i+1}\\
t_{b_i}=t_{a_i+1}
\end{cases}
\text{or}
\begin{cases}t_{a_i}=t_{b_i}\\ t_{a_i+1}=t_{b_i+1}
\end{cases}\mathrm{for  } \  1\le i \le k\Big\}.
\end{eqnarray*}
For $\textbf t=(t_1,...,t_{2k})\in A(\pi)$, let $r(\pi,\textbf
t)=\sharp\{i\in[k]|t_{a_i}=t_{b_i}=t_{a_i+1}=t_{b_i+1}\},$ then
$\frac{1}{m}\langle
\mathrm{tr}(H/\sqrt{m})^{2k}\rangle_{\mathrm{GOE}}=m^{-k-1}\sum\limits_{\pi\in\mathcal{P}_2(2k)}\sum\limits_{\textbf
t\in A(\pi)}2^{r(\pi,\textbf t)}.$

Similarly as in the proof of Theorem 3.2 and Theorem 4.3 of
\cite{llw}, we know when $k$ is odd $\frac{1}{mN}\langle
\mathrm{tr}X_N^{k} \rangle_{\mathrm{TBM}}=o(1)$ and
$\frac{1}{mN}\langle  \mathrm{tr}X_N^{2k} \rangle_{\mathrm{TBM}}$ is
\begin{eqnarray*}
&
&\sum\limits_{i=1}^N\sum\limits_{j_1,...,j_{2k}=-b_N}^{b_N}\sum\limits_{t_1,...,t_{2k}=1}^m\frac{E(a_{t_1t_2}(j_1)\cdots
a_{t_{2k}t_1}(j_{2k}))}{(2mb_N)^k\cdot mN}\prod_{l=1}^{2k}
I_{[1,N]}(i+\sum\limits_{q=1}^l
j_q)\delta_{0,\sum\limits_{q=1}^{2k}j_q}\\
&=&\sum\limits_{i=1}^N\sum\limits_{\pi\in\mathcal{P}_2(2k)}\sum\limits_{\textbf
t\in A(\pi)}\frac{2^{r(\pi,\textbf t)}}{{(2mb_N)^k\cdot
mN}}\sum\limits_{x_1,...,x_k=-b_N} ^{b_N}\prod_{l=1}^{2k}
I_{[1,N]}(i+\sum_{q=1}^l \epsilon_{\pi}(q)x_{\pi(q)})+o(1)\\
&\to&\sum\limits_{\pi\in\mathcal{P}_2(2k)}\sum\limits_{\textbf t\in
A(\pi)}\frac{2^{r(\pi,\textbf t)}}{m^{k+1}}\,(N\to\infty).
\end{eqnarray*}
Thus $ \lim\limits_{N\to\infty}\frac{1}{mN}\langle
\mathrm{tr}X_N^{k} \rangle_{\mathrm{TBM}} =\frac{1}{m}\langle
\mathrm{tr}(H/\sqrt{m})^{k} \rangle_{\mathrm{GOE}}.$ From \cite{bb}
we know the convergence is almost sure if $\sum_{N=1}^{\infty}
b_{N}^{-2}<\infty$. Finally, $g_m(x)$ is the one-point correlation
function of GOE with order $m$, thus we complete the proof.

\end{proof}

\begin{remark}The density (\ref{goedensity}) follows from the 1-point correlation
function of GOE (see (7.2.32) of \cite{m}). This family densities
can also be seen as a transition from the  normal distribution
$N(0,2)$ to the semicircle distribution with variance 1. In the
situation of GOE, we also have parallel results to Theorem
\ref{tracepolynomial}.
\end{remark}

\section*{Acknowledgements}
The authors thank the anonymous referee for telling us that
\cite{kkm} has a similar result to Theorem \ref{thm1}.

\end{document}